\documentclass[11pt]{amsart}

\usepackage{amsmath,amssymb,amsthm}
\usepackage[margin=1.15in]{geometry}
\usepackage[colorlinks=true,linkcolor=blue,citecolor=blue,urlcolor=blue]{hyperref}

\newtheorem{theorem}{Theorem}[section]
\newtheorem{lemma}[theorem]{Lemma}
\newtheorem{proposition}[theorem]{Proposition}
\newtheorem{corollary}[theorem]{Corollary}
\newtheorem{question}[theorem]{Question}
\theoremstyle{definition}
\newtheorem{definition}[theorem]{Definition}
\theoremstyle{remark}
\newtheorem{remark}[theorem]{Remark}

\numberwithin{equation}{section}

\newcommand{\R}{\mathbb{R}}
\newcommand{\Q}{\mathbb{Q}}
\newcommand{\Z}{\mathbb{Z}}
\newcommand{\N}{\mathbb{N}}
\newcommand{\E}{\mathbb{E}}
\newcommand{\I}{\mathcal{I}}
\newcommand{\B}{\mathcal{B}}

\begin{document}

\title[Solyanik estimates for ergodic flows]{Solyanik estimates for ergodic flows}

\author{Kesen Huang}
\thanks{This paper was prepared with substantial assistance from Claude (Anthropic), which proposed the main arguments. The author completed the gaps in the proofs and independently verified all results; any remaining errors are the author's.}
\address{Department of Mathematics, Baylor University, Waco, Texas}
\email{Kesen\_Huang1@baylor.edu}


\begin{abstract}
Let $\{T_t\}_{t\in\R}$ be a measure-preserving flow on a probability space $(\Omega,\Sigma,\mu)$ and let $T^*$ be the associated one-sided ergodic maximal operator. We show that the sharp Tauberian constant
\[
C(\alpha)=\sup_{\mu(E)>0}\frac{\mu(\{T^*\chi_E>\alpha\})}{\mu(E)},\qquad 0<\alpha<1,
\]
equals $1/\alpha$ for every $\alpha\in(0,1)$, unless every measurable set is invariant modulo null sets, in which case $C(\alpha)=1$ for every $\alpha$. In particular the Solyanik estimate $\lim_{\alpha\to1^-}C(\alpha)=1$ holds for every measure-preserving flow, and for ergodic flows on nonatomic spaces the estimate $C(\alpha)-1\le 1/\alpha-1$ holds with equality. This answers the one-parameter case of a question of Hagelstein and Parissis. The main point is the lower bound, which is obtained from a lacunary cross-section built out of the upcrossings of a probe function along the orbits of the flow.
\end{abstract}

\maketitle

\section{Solyanik estimates}\label{s:intro}

Let $M$ be a maximal operator acting on locally integrable functions on $\R^n$, such as the uncentered Hardy--Littlewood maximal operator
\[
M_{HL}f(x)=\sup_{B\ni x}\frac{1}{|B|}\int_B|f|,
\]
where the supremum is taken over all Euclidean balls containing $x$, its centered variant $M_{HL,c}$, or the strong maximal operator $M_S$, for which the supremum is taken over all rectangles containing $x$ with sides parallel to the coordinate axes. The \emph{sharp Tauberian constant} of $M$ with respect to $\alpha\in(0,1)$ is
\[
C_M(\alpha)=\sup_{0<|E|<\infty}\frac{1}{|E|}\bigl|\{x\in\R^n:\ M\chi_E(x)>\alpha\}\bigr| ;
\]
as a function of $\alpha$ it is also known as the \emph{halo function} of the associated differentiation basis. The classical restricted weak-type estimates give, for instance, $C_{M_{HL}}(\alpha)\le 3^n/\alpha$, a bound which stays away from $1$ as $\alpha\to1^-$. On the other hand, when $\alpha$ is close to $1$, the set $\{M\chi_E>\alpha\}$ consists of points lying in a ball or rectangle in which $E$ has relative density close to $1$, and it is natural to expect that this set is not much larger than $E$ itself.

This expectation was first confirmed by A.~A.~Solyanik~\cite{Sol93}, who studied the behavior of the halo functions of the Hardy--Littlewood and strong maximal operators as $\alpha\to1^-$. The following quantitative statements are due to Solyanik~\cite{Sol93} and Hagelstein--Parissis~\cite{HP14}: for $\alpha$ sufficiently close to $1$,
\begin{equation}\label{e:euclid}
C_{M_{HL,c}}(\alpha)-1\lesssim_n \frac1\alpha-1,\qquad
C_{M_{HL}}(\alpha)-1\lesssim_n \Bigl(\frac1\alpha-1\Bigr)^{\frac{1}{n+1}},\qquad
C_{M_S}(\alpha)-1\lesssim_n \Bigl(\frac1\alpha-1\Bigr)^{\frac1n}.
\end{equation}
In particular $\lim_{\alpha\to1^-}C_M(\alpha)=1$ for each of these operators. Estimates of this type are now called \emph{Solyanik estimates}. We refer to the survey~\cite{HP14} for background, for a discussion of the optimal exponents in~\eqref{e:euclid}, which are not known in general, and for further references. Solyanik estimates are also useful beyond their intrinsic interest: they were used in~\cite{HP15} to show that the halo functions $C_{M_{HL}}$ and $C_{M_S}$ belong to the H\"older class $C^{1/n}(0,1)$; see also~\cite{BH14} for related continuity results for halo functions.

\section{Solyanik estimates in ergodic theory}\label{s:ergodic}

There is a well-known correspondence between the distribution functions of geometric maximal operators in harmonic analysis and those of their ergodic-theoretic counterparts, going back to Calder\'on's transference principle~\cite{Cal68}. It is therefore natural to ask whether Solyanik estimates have ergodic analogues.

For a single measure-preserving transformation $T$ of a probability space $(\Omega,\Sigma,\mu)$ the answer has been known since the early days of ergodic theory. Indeed, with
\[
T^*f(\omega)=\sup_{N\ge1}\frac1N\sum_{j=0}^{N-1}|f(T^j\omega)|,
\]
the maximal ergodic theorem of Wiener~\cite{Wie39} and Yosida--Kakutani~\cite{YK39} gives $\mu(\{T^*f>\alpha\})\le\alpha^{-1}\|f\|_{L^1(\mu)}$. Applied to $f=\chi_E$ this yields
\[
\sup_{\mu(E)>0}\frac{\mu(\{T^*\chi_E>\alpha\})}{\mu(E)}\le\frac1\alpha\longrightarrow1\qquad(\alpha\to1^-).
\]
This argument does not iterate to give Solyanik estimates for multiparameter ergodic maximal operators. In~\cite{HP16}, Hagelstein and Parissis considered commuting measure-preserving transformations $U_1,\dots,U_n$ of $(\Omega,\Sigma,\mu)$ and the ergodic maximal operators
\[
M^*_{\mathcal B}f(\omega)=\sup_{R\in\mathcal B}\frac{1}{\#(R\cap\Z^n)}\sum_{(j_1,\dots,j_n)\in R\cap\Z^n}\bigl|f(U_1^{j_1}\cdots U_n^{j_n}\omega)\bigr|,
\]
where $\mathcal B$ is the collection of open rectangles with sides parallel to the axes containing the origin (giving the ergodic strong maximal operator $M^*_S$), the collection of open balls containing the origin ($M^*_{HL}$), or the collection of open balls centered at the origin ($M^*_{HL,c}$). Writing $C^*_S$, $C^*_{HL}$ and $C^*_{HL,c}$ for the corresponding sharp Tauberian constants
\[
C^*_{\mathcal B}(\alpha)=\sup_{\mu(E)>0}\frac{\mu(\{M^*_{\mathcal B}\chi_E>\alpha\})}{\mu(E)},
\]
they proved, by means of a transference principle in the spirit of Calder\'on combined with the Euclidean estimates~\eqref{e:euclid}, that for $\alpha$ sufficiently close to $1$
\begin{equation}\label{e:HP16}
C^*_S(\alpha)-1\lesssim_n\Bigl(\frac1\alpha-1\Bigr)^{\frac1n},\qquad
C^*_{HL,c}(\alpha)-1\lesssim_n\frac1\alpha-1,\qquad
C^*_{HL}(\alpha)-1\lesssim_n\Bigl(\frac1\alpha-1\Bigr)^{\frac{1}{n(n+1)}}.
\end{equation}
In~\cite{HP17} the same authors studied the regularity of the Tauberian constants associated with discrete and ergodic strong maximal operators.

The results of~\cite{HP16} concern ergodic averages given by discrete sums. Among the problems raised there~\cite[Problem~7.1]{HP16} is whether Solyanik estimates also hold for maximal operators associated with measurable flows. The simplest instance of this problem is the following.

\begin{question}\label{q:main}
Let $\{T_t\}_{t\in\R}$ be an ergodic measure-preserving flow on a probability space $(\Omega,\Sigma,\mu)$, as in~\cite[p.~76]{Pet83}, and let
\[
T^*f(\omega)=\sup_{s>0}\frac1s\int_0^s|f(T_t\omega)|\,dt,\qquad
C(\alpha)=\sup_{\substack{E\in\Sigma\\ \mu(E)>0}}\frac{\mu(\{\omega\in\Omega:\ T^*\chi_E(\omega)>\alpha\})}{\mu(E)} .
\]
Must $\lim_{\alpha\to1^-}C(\alpha)=1$?
\end{question}

In this note we answer Question~\ref{q:main} in the affirmative. In fact we compute $C(\alpha)$ exactly for every $\alpha\in(0,1)$, and we do so for every measure-preserving flow, ergodic or not.

\section{Main results}\label{s:main}

Throughout, $(\Omega,\Sigma,\mu)$ is a probability space and $\{T_t\}_{t\in\R}$ is a \emph{measure-preserving flow}: each $T_t\colon\Omega\to\Omega$ is measurable and preserves $\mu$(i.e. $\mu\circ T_t=\mu$ for all $t$), $T_0=\mathrm{id}$, $T_{s+t}=T_s\circ T_t$ for all $s,t\in\R$, and the map $(\omega,t)\mapsto T_t\omega$ is $\Sigma\otimes\B(\R)$-measurable, where $\B(\R)$ denotes the Borel $\sigma$-algebra of $\R$. We write $|\cdot|$ for Lebesgue measure on $\R$. The maximal operator $T^*$ and the constants $C(\alpha)$ are as in Question~\ref{q:main}. For bounded measurable $f$ the function $s\mapsto\frac1s\int_0^s|f(T_t\omega)|\,dt$ is continuous on $(0,\infty)$, so the supremum defining $T^*f$ may be restricted to rational $s$; in particular $T^*\chi_E$ is measurable.

Let
\[
\I=\{B\in\Sigma:\ \mu(T_tB\,\triangle\,B)=0\ \text{ for every } t\in\R\}
\]
be the $\sigma$-algebra of sets that are invariant modulo null sets. It contains every $\mu$-null set. The flow is \emph{ergodic} if $\mu(B)\in\{0,1\}$ for every $B\in\I$; by proof~\ref{t:birkhoff} below this agrees with the definition via strictly invariant sets.

\begin{theorem}\label{t:main}
Let $\{T_t\}_{t\in\R}$ be a measure-preserving flow on a probability space $(\Omega,\Sigma,\mu)$. Then
\begin{equation}\label{e:solyanik}
1\le C(\alpha)\le\frac1\alpha\qquad\text{for every }\alpha\in(0,1).
\end{equation}
Moreover, exactly one of the following holds:
\begin{enumerate}
\item[(a)] $\I=\Sigma$, that is, every measurable set is invariant modulo null sets. Then $C(\alpha)=1$ for every $\alpha\in(0,1)$.
\item[(b)] $\I\ne\Sigma$. Then $C(\alpha)=1/\alpha$ for every $\alpha\in(0,1)$.
\end{enumerate}
\end{theorem}

\begin{corollary}\label{c:ergodic}
Let $\{T_t\}_{t\in\R}$ be an ergodic measure-preserving flow on a nonatomic probability space. Then $C(\alpha)=1/\alpha$ for every $\alpha\in(0,1)$. In particular $\lim_{\alpha\to1^-}C(\alpha)=1$, and the Solyanik estimate
\[
C(\alpha)-1\le\frac1\alpha-1
\]
holds with equality for every $\alpha\in(0,1)$.
\end{corollary}

\begin{proof}
Since $\mu$ is nonatomic, there is $A\in\Sigma$ with $0<\mu(A)<1$. By ergodicity $A\notin\I$, so case~(b) of Theorem~\ref{t:main} applies.
\end{proof}

\begin{remark}\label{r:main}
\begin{enumerate}
\item The estimate~\eqref{e:solyanik} already answers Question~\ref{q:main}, and it does not require ergodicity: the upper bound is the maximal ergodic theorem for flows, and the lower bound is trivial(by taking $E=\Omega$). The substance of Theorem~\ref{t:main} is the lower bound $C(\alpha)\ge1/\alpha$ in case~(b), which shows that the Solyanik estimate cannot be improved for any flow that moves some set.
\item The value $1/\alpha$ is the Euclidean one. For the one-sided maximal operator\\
\mbox{$M^+f(x)=\sup_{s>0}\frac1s\int_x^{x+s}|f|$} on $\R$, the rising sun lemma gives\\
\mbox{$|\{M^+\chi_E>\alpha\}|=\alpha^{-1}|E\cap\{M^+\chi_E>\alpha\}|$}, and a.e.\ point of $E$
lies in $\{M^+\chi_E>\alpha\}$ by Lebesgue's differentiation theorem; hence
\mbox{$|\{M^+\chi_E>\alpha\}|=|E|/\alpha$} for every $E$ with $0<|E|<\infty$. For instance,
for $E=[0,1]$ the level set is the interval $(1-1/\alpha,1)$. The proof of the lower
bound in case~(b) transplants this configuration into $\Omega$ along the orbits of the flow. 
\item If the supremum defining $C(\alpha)$ is restricted to sets with $0<\mu(E)<1$, Theorem~\ref{t:main} remains true, provided such sets exist: the extremal sets constructed in Section~\ref{s:proof} satisfy $\mu(E)\le\alpha$.
\item The two-sided maximal operator behaves analogously, with $1/\alpha$ replaced by $2/\alpha-1$; see Proposition~\ref{p:twosided}.
\end{enumerate}
\end{remark}

\section{Proof of \texorpdfstring{Theorem~\ref{t:main}}{the main theorem}}\label{s:proof}

\subsection{Preliminaries}

We shall use two classical facts about flows; see~\cite{Pet83}.

\begin{theorem}[Maximal ergodic theorem for flows~\cite{Pet83},P76]\label{t:max}
For every $f\in L^1(\mu)$ and $\alpha>0$,\\
 $\mu(\{T^*f>\alpha\})\le\alpha^{-1}\|f\|_{L^1(\mu)}$.
\end{theorem}

\begin{theorem}[Birkhoff's ergodic theorem for flows]\label{t:birkhoff}
For every $f\in L^1(\mu)$,
\[
\lim_{L\to\infty}\frac1L\int_0^Lf(T_t\omega)\,dt=\E[f\mid\I](\omega)\qquad\text{for $\mu$-a.e. }\omega\in\Omega.
\]
\end{theorem}

\begin{proof}\label{r:invariant}
    We give the proof of Birkhoff's ergodic theorem for flows for completeness. \par
    We do almost the same thing as the proof of the discrete time. We proof:
    \begin{itemize}
        \item [(1)] $\lim_{L\to\infty}\frac1L\int_0^Lf(T_t\omega)\,dt=\overline{f}\,\text{ a.e and in $L^1$}$;
        \item [(2)] $\overline{f}$ is $\I$ measurable, and $\overline{f}=\E[f\mid\I]$.
    \end{itemize}
    (1): Since by discrete time version of Birkhoff's ergodic theorem, we have 
    \begin{align*}
        \frac{1}{L}\int_0^L f(T_t\omega)dt=\frac{1}{L}\int_0^{\left\lfloor L\right\rfloor}  f(T_t\omega)dt+\frac{1}{L}\int_{\left\lfloor L\right\rfloor }^L f(T_t\omega)dt.
    \end{align*}
    The first integral goes to $\overline f(\omega)$, by considering $F(\omega)=\int_0^1 f(T_t\omega)dt$, the first intergral
    \begin{align*}
        \frac{1}{L}\int_0^{\left\lfloor L\right\rfloor}  f(T_t\omega)dt=&\frac{\left\lfloor L\right\rfloor}{L}\frac{1}{\left\lfloor L\right\rfloor}\sum_{j=0}^{\left\lfloor L\right\rfloor-1}\int_{j}^{j+1}f(T_t\omega)dt\\
        =&\frac{\left\lfloor L\right\rfloor}{L}\frac{1}{\left\lfloor L\right\rfloor}\sum_{j=0}^{\left\lfloor L\right\rfloor-1}\int_{0}^{1}f(T_tT_j\omega)dt\\
        =&\frac{\left\lfloor L\right\rfloor}{L}\frac{1}{\left\lfloor L\right\rfloor}\sum_{j=0}^{\left\lfloor L\right\rfloor-1}F(T_j\omega)dt\to \overline f(\omega).
    \end{align*}
    For the second integral, we have estimation below:
    \begin{align*}
        \left\lvert \frac{1}{L}\int_{\left\lfloor L\right\rfloor }^L f(T_t\omega)dt\right\rvert \leq& \frac{1}{L}\int_{\left\lfloor L\right\rfloor }^{\left\lfloor L\right\rfloor +1} \left\lvert f(T_t\omega)\right\rvert dt\\
        =&\frac{\left\lfloor L\right\rfloor }{L}\frac{1}{\left\lfloor L\right\rfloor }\int_{\left\lfloor L\right\rfloor }^{\left\lfloor L\right\rfloor +1} \left\lvert f(T_t\omega)\right\rvert dt\\
        =&\frac{\left\lfloor L\right\rfloor }{L}\left(\frac{\left\lfloor L\right\rfloor+ 1}{\left\lfloor L\right\rfloor }\frac{1}{\left\lfloor L\right\rfloor +1}\int_0^{\left\lfloor L\right\rfloor +1}\left\lvert f(T_t\omega)\right\rvert dt-\frac{ 1}{\left\lfloor L\right\rfloor }\int_0^{\left\lfloor L\right\rfloor }\left\lvert f(T_t\omega)\right\rvert dt\right)\to 0.   \\
    \end{align*}\par
    For $L^1$ case, we see if $f$ is bounded, then $\frac{1}{L}\int_0^L f(T_t\omega)dt\to \overline f$ in $L^1$, now assume $f\in L^1$, we can approximate $f$ with a bounded function, say $g$, $\left\lVert g-f\right\rVert_{1} <\epsilon$.
    {\small \begin{align*}
        \left\lVert \frac1L\int_0^Lf(T_t\omega)\,dt-\overline{f}\right\rVert_1\leq  \left\lVert \frac1L\int_0^Lf(T_t\omega)\,dt-\frac1L\int_0^Lg(T_t\omega)\,dt\right\rVert_1+\left\lVert \frac1L\int_0^Lg(T_t\omega)\,dt-\overline g\right\rVert_1 +\left\lVert \overline g-\overline f\right\rVert_1,
    \end{align*}}
    by Fubini's theorem and Dominated Convergence Theorem, we see the first two terms on right hand side are small. And $\overline g-\overline f=\overline{ g-f}$, by Fubini's theorem and Fatou's lemma,
    \begin{align*}
        \left\lVert \overline{f}\right\rVert_1 =&\left\lVert \liminf_{L\to\infty}\left\lvert \frac1L\int_0^Lf(T_t\omega)\,dt\right\rvert \right\rVert_1\leq \liminf_{L\to\infty}\left\lVert\frac1L\int_0^Lf(T_t\omega)\,dt\right\rVert_1\\
        \leq &\liminf_{L\to\infty}\frac1L\int_{\Omega}\int_0^L\left\lvert f(T_t\omega)\right\rvert \,dt\, d\mu=\liminf_{L\to\infty}\frac1L\int_0^L\int_{\Omega}\left\lvert f(T_t\omega)\right\rvert\, d\mu \,dt=\left\lVert f\right\rVert _1.
    \end{align*}
    So $\left\lVert \overline g-\overline f\right\rVert_1\leq \left\lVert {g-f}\right\rVert_1<\epsilon$ is small. Then $\left\lVert \overline f -f \right\rVert_1\leq \left\lVert \overline f -\overline g \right\rVert_1+ \left\lVert \overline g -g \right\rVert_1+ \left\lVert g -f \right\rVert_1<3\epsilon$. \par
    (2): $\overline f$ is $\I$ measurable, we prove this by showing for any $s$, $\overline f(T_s)=\overline f$ a.e..
    \begin{align*}
        \overline f(T_s\omega)=&\lim_{L\to \infty} \frac1L\int_0^Lf(T_tT_s\omega)\,dt=\lim_{L\to \infty} \frac1L\int_0^Lf(T_{t+s}\omega)\,dt\\
        =&\lim_{L\to \infty} \frac1L\int_{s}^{L+s}f(T_{t}\omega)\,dt=\lim_{L\to \infty} \frac1L\left(\int_{0}^{L+s}f(T_{t}\omega)\,dt-\int_{0}^{s}f(T_{t}\omega)\,dt\right)\\
        =&\overline{f}(\omega).
    \end{align*}
    So $T_s^{-1}\{\omega:\overline f(\omega)>a\}=T_s^{-1}\{\omega:\overline f(T_s\omega)>a\}=\{\omega:\overline f(\omega)>a\}$ a.e. and in this way $\{\omega:\overline f(\omega)>a\}\in \I$.\par
    For any $B\in \I$, we prove $\int_B \overline f=\int_B f$. Let $F(\omega)=\int_{0}^{1} f(T_t\omega)\,dt$, so 
    \begin{align*}
        \int_B \frac{1}{N}\sum_{k=0}^{N-1}F(T_k\omega)\,d\mu=\int_B \frac{1}{N}\sum_{k=0}^{N-1}\int_0^1 f(T_{t+k}\omega)\,dt \,d\mu=\int_B\frac{1}{N}\int_0^N f(T_t\omega)\,dt\,d\mu\to \int_B \overline f\,d\mu.
    \end{align*}
    On the other hand, 
    \begin{align*}
        \int_B F(T_k\omega)\,d\mu=\int_{T_k^{-1}B}F(\omega)\,d\mu=\int_{B}F(\omega)\,d\mu.
    \end{align*}
    So 
    \[\int_{B}F(\omega)\,d\mu=\int_B \overline f\,d\mu.\]
    By Fubini's theorem, we have:
    \[\int_{B}F(\omega)\,d\mu=\int_B\int_0^1 f(T_t\omega)\, dt \, d\mu=\int_0^1\int_B f(T_t\omega)\, d\mu\, dt=\int_0^1\int_{T_t^{-1}B} f(\omega)\, d\mu\, dt=\int_B f\, d\mu. \]
    So $\int_B \overline f\,d\mu=\int_B f\, d\mu$.\par 
    This also means we can establish Birkhoff's ergodic theorem for flows with $\I$, instead of strict invariant sets.
\end{proof}

The reversed flow $\{T_{-t}\}_{t\in\R}$ is again a measure-preserving flow, with the same $\sigma$-algebra $\I$, so Theorems~\ref{t:max} and~\ref{t:birkhoff} apply to it as well.

\begin{lemma}[Differentiation along orbits]\label{l:diff}
Let $E\in\Sigma$. For $\mu$-a.e.\ $\omega\in\Omega$,
\[
\lim_{n\to\infty}n\int_0^{1/n}\chi_E(T_s\omega)\,ds=\chi_E(\omega).
\]
\end{lemma}

\begin{proof}
Let $D$ be the set of $\omega$ for which this fails; $D\in\Sigma$. Fix $\omega$. By joint measurability, $\varphi_\omega(s)=\chi_E(T_s\omega)$ is a measurable function on $\R$, and by the group property
\[
n\int_0^{1/n}\chi_E(T_sT_t\omega)\,ds=n\int_t^{t+1/n}\varphi_\omega(s)\,ds .
\]
Hence $T_t\omega\in D$ exactly when $n\int_t^{t+1/n}\varphi_\omega\not\to\varphi_\omega(t)$, and by Lebesgue's differentiation theorem the set of such $t$ is Lebesgue-null. Since $(\omega,t)\mapsto\chi_D(T_t\omega)$ is jointly measurable and each $T_t$ preserves $\mu$, Tonelli's theorem gives
\[
\mu(D)=\int_0^1\mu(T_t^{-1}D)\,dt=\int_\Omega\bigl|\{t\in[0,1]:\ T_t\omega\in D\}\bigr|\,d\mu(\omega)=0. \qedhere
\]
\end{proof}

\subsection{The bounds \texorpdfstring{\eqref{e:solyanik}}{(3.1)} and case~(a)}\label{ss:casea}

The upper bound $C(\alpha)\le1/\alpha$ follows from Theorem~\ref{t:max} applied to $f=\chi_E$. For the lower bound take $E=\Omega$: then $T^*\chi_\Omega\equiv1>\alpha$, so the ratio equals $1$. (More precisely, Lemma~\ref{l:diff} shows that $\mu(\{T^*\chi_E>\alpha\})\ge\mu(E)$ for every $E\in\Sigma$, since a.e.\ point of $E$ admits short windows on which the average of $\chi_E$ along its orbit exceeds $\alpha$.)

Assume now that $\I=\Sigma$, and let $E\in\Sigma$. Since $E\in\I$, for each $t$ we have $\chi_E(T_t\omega)=\chi_E(\omega)$ for a.e.\ $\omega$. The set $\{(\omega,t):\ \chi_E(T_t\omega)\ne\chi_E(\omega)\}$ is $\Sigma\otimes\B(\R)$-measurable, so by Tonelli's theorem, for a.e.\ $\omega$ we have $\chi_E(T_t\omega)=\chi_E(\omega)$ for a.e.\ $t\ge0$. For such $\omega$ every average in the definition of $T^*\chi_E(\omega)$ equals $\chi_E(\omega)$. Hence $T^*\chi_E=\chi_E$ a.e., so $\mu(\{T^*\chi_E>\alpha\})=\mu(E)$, and $C(\alpha)=1$. This proves~(a).

\subsection{Case (b): a probe function}

Assume from now on that $\I\ne\Sigma$, and fix $A\in\Sigma\setminus\I$. In view of~\eqref{e:solyanik}, it suffices to show that for each $\alpha\in(0,1)$ there is $E\in\Sigma$ with $\mu(E)>0$ and
\begin{equation}\label{e:goal}
\mu(\{T^*\chi_E>\alpha\})\ge\frac{\mu(E)}{\alpha}.
\end{equation}

For $h>0$ define
\[
G_h(\omega)=\frac1h\int_0^h\chi_A(T_s\omega)\,ds,\qquad\omega\in\Omega .
\]
By Tonelli's theorem $G_h$ is measurable, and $0\le G_h\le1$. For $\omega\in\Omega$ and $t\in\R$ put $g_\omega(t)=G_h(T_t\omega)$. By the group property,
\[
g_\omega(t)=\frac1h\int_t^{t+h}\chi_A(T_s\omega)\,ds .
\]

\begin{lemma}\label{l:lip}
For every $\omega\in\Omega$ and all $t,t'\in\R$ we have $|g_\omega(t)-g_\omega(t')|\le|t-t'|/h$.
\end{lemma}

\begin{proof}
Let $t<t'$. Cancelling the common part of the two intervals of integration,
\[
g_\omega(t')-g_\omega(t)=\frac1h\Bigl(\int_{t+h}^{t'+h}\chi_A(T_s\omega)\,ds-\int_t^{t'}\chi_A(T_s\omega)\,ds\Bigr),
\]
and each of the two integrals lies in $[0,t'-t]$.
\end{proof}

\begin{lemma}\label{l:noninv}
There is $h>0$ such that $G_h$ is not $\I$-measurable.
\end{lemma}

\begin{proof}
Suppose that $G_{1/n}$ is $\I$-measurable for every $n\in\N$. By Lemma~\ref{l:diff}, $G_{1/n}\to\chi_A$ $\mu$-a.e. Since $\I$ contains all null sets, an a.e.\ limit of $\I$-measurable functions is $\I$-measurable; hence $\chi_A$ is $\I$-measurable, that is, $A\in\I$, a contradiction.
\end{proof}

We fix $h$ as in Lemma~\ref{l:noninv} and write $G=G_h$; from now on $g_\omega$ refers to this $G$.

\begin{lemma}\label{l:levels}
There exist rational numbers $a<b$ and a set $P\in\Sigma$ with $\mu(P)>0$ such that
\[
\E[\chi_{\{G<a\}}\mid\I]>0\quad\text{and}\quad\E[\chi_{\{G>b\}}\mid\I]>0\qquad\text{on }P.
\]
\end{lemma}

\begin{proof}
For $q\in\Q$ fix nonnegative $\I$-measurable versions $\varphi_q$ of $\E[\chi_{\{G<q\}}\mid\I]$ and $\psi_q$ of $\E[\chi_{\{G>q\}}\mid\I]$, and define the $\I$-measurable functions
\[
m(\omega)=\sup\{q\in\Q:\ \varphi_q(\omega)=0\},\qquad M(\omega)=\inf\{q\in\Q:\ \psi_q(\omega)=0\},
\]
with values in $[-\infty,+\infty]$. For $q\in\Q$ let $S_q=\{\varphi_q=0\}\in\I$. By the defining property of conditional expectation,
\[
\mu(S_q\cap\{G<q\})=\int_{S_q}\chi_{\{G<q\}}\,d\mu=\int_{S_q}\varphi_q\,d\mu=0 .
\]
If $G(\omega)<m(\omega)$, there is $q\in\Q$ with $q>G(\omega)$ and $\varphi_q(\omega)=0$, so that $\omega\in S_q\cap\{G<q\}$. Hence $\{G<m\}$ is contained in a countable union of null sets, and $G\ge m$ a.e. Symmetrically, $G\le M$ a.e.

If $m=M$ a.e., then $G=m$ a.e., and $G$ would be $\I$-measurable (again because $\I$ contains the null sets), contradicting Lemma~\ref{l:noninv}. Thus $\mu(\{m<M\})>0$. Since
\[
\{m<M\}=\bigcup_{a,b\in\Q,\ a<b}\bigl(\{m<a\}\cap\{b<M\}\bigr),
\]
there are rationals $a<b$ such that $P=\{m<a\}\cap\{b<M\}$ has positive measure. If $\omega\in P$, then $a>m(\omega)$, so $\varphi_a(\omega)\ne0$, that is, $\varphi_a(\omega)>0$; likewise $b<M(\omega)$ gives $\psi_b(\omega)>0$.
\end{proof}

We fix $a<b$ and $P$ as in Lemma~\ref{l:levels} and set
\[
\ell=(b-a)h>0 .
\]
Note that $P\in \I$, since $m$ and $M$ are $\I$-measurable; in particular, if the flow is ergodic, then $\mu(P)=1$.

\subsection{Upcrossings}

\begin{definition}
Let $\omega\in\Omega$. A time $t\in\R$ is an \emph{upcrossing time} of $\omega$ if $g_\omega(t)=b$ and there is $s<t$ such that $g_\omega(s)\le a$ and $g_\omega(u)<b$ for all $u\in[s,t)$. We denote by $R_\omega\subset\R$ the set of upcrossing times of $\omega$.
\end{definition}

Thus $t\in R_\omega$ when, at time $t$, the function $g_\omega$ completes a passage from the level $a$ up to the level $b$. Merely reaching the level $b$ does not suffice; this two-level condition is what makes the upcrossing times well separated.

\begin{lemma}\label{l:R}
\begin{enumerate}
\item[(R1)] $R_{T_x\omega}=R_\omega-x$ for all $\omega\in\Omega$ and $x\in\R$.
\item[(R2)] Distinct points of $R_\omega$ are at distance at least $2\ell$; in particular $R_\omega$ is locally finite.
\item[(R3)] For $\mu$-a.e.\ $\omega\in P$ we have $R_\omega\cap(-\infty,0]\ne\varnothing$.
\end{enumerate}
\end{lemma}

\begin{proof}
(R1) follows from the identity $g_{T_x\omega}(t)=g_\omega(t+x)$.

(R2) Let $t<t'$ be points of $R_\omega$, and let $s'<t'$ be as in the definition for $t'$. Since $g_\omega(t)=b$ while $g_\omega<b$ on $[s',t')$, we have $t<s'$. By Lemma~\ref{l:lip}, passing from the value $b$ at time $t$ to a value $\le a$ at time $s'$ takes time at least $(b-a)h=\ell$, and passing from a value $\le a$ at time $s'$ to the value $b$ at time $t'$ takes time at least $\ell$. Hence $t'-t\ge2\ell$.

(R3) Applying Theorem~\ref{t:birkhoff} to the reversed flow with $f=\chi_{\{G<a\}}$ and with $f=\chi_{\{G>b\}}$, we obtain, for a.e.\ $\omega$,
\begin{align*}
\frac1L\bigl|\{t\in[0,L]:\ g_\omega(-t)<a\}\bigr|&\to\E[\chi_{\{G<a\}}\mid\I](\omega),\\
\frac1L\bigl|\{t\in[0,L]:\ g_\omega(-t)>b\}\bigr|&\to\E[\chi_{\{G>b\}}\mid\I](\omega)
\end{align*}
as $L\to\infty$($\{t\geq 0: g_\omega(-t)<a\}=\{t\geq 0: G(T_{-t}\omega)<a\}=\{t:T_{-t}\omega\in \{G<a\}\}$, so $\frac{1}{L}\int_0^L\chi_{\{G<a\}}(T_{-t}\omega)\,dt=\frac{1}{L}\left\lvert \left\{t\in [0,L]:T_{-t}\omega\in \{G<a\}\right\} \right\rvert =\frac1L\left\lvert \{t\in[0,L]:\ g_\omega(-t)<a\}\right\rvert $). For a.e.\ $\omega\in P$ both limits are positive, so the sets $\{t\ge0:\ g_\omega(-t)<a\}$ and $\{t\ge0:\ g_\omega(-t)>b\}$ are both unbounded, since a subset of $[0,K]$ occupies at most the proportion $K/L$ of $[0,L]$. Fix such an $\omega$, choose $t_2\ge0$ with $g_\omega(-t_2)>b$ and then $t_1>t_2$ with $g_\omega(-t_1)<a$, and set
\[
u^*=\inf\{u\ge-t_1:\ g_\omega(u)\ge b\},\qquad s^*=\max\{s\in[-t_1,u^*]:\ g_\omega(s)\le a\}.
\]
By continuity, $u^*\in(-t_1,-t_2]$, $g_\omega(u^*)=b$, $g_\omega(s^*)\le a$, and $g_\omega<b$ on $[s^*,u^*)$. Hence $u^*\in R_\omega\cap(-\infty,0]$.
\end{proof}

\subsection{The height function}

For $\omega\in\Omega$ let
\[
t_0(\omega)=\min\{t\ge0:\ -t\in R_\omega\}\in[0,\infty],
\]
with $t_0(\omega)=\infty$ if $R_\omega\cap(-\infty,0]=\varnothing$; the minimum exists by~(R2). Thus $t_0(\omega)$ is the time elapsed since the last upcrossing, and we call it the \emph{height} of $\omega$.

\begin{lemma}[Sawtooth lemma]\label{l:saw}
Let $\omega\in\Omega$ with $t_0(\omega)=t<\infty$.
\begin{enumerate}
\item[(a)] If $0\le x<2\ell-t$, then $t_0(T_x\omega)=t+x$.
\item[(b)] If $0\le x\le t$, then $t_0(T_{-x}\omega)=t-x$.
\end{enumerate}
\end{lemma}

\begin{proof}
(a) By~(R1), $-(t+x)\in R_{T_x\omega}$, and the points of $R_{T_x\omega}$ in $(-(t+x),0]$ correspond to the points of $R_\omega$ in $(-t,x]$. Since $-t\in R_\omega$, by~(R2) there are no points of $R_\omega$ in $(-t,-t+2\ell)$, and this interval contains $(-t,x]$ because $x<2\ell-t$. Hence $R_{T_x\omega}\cap (-(x+t),0]=\emptyset$. So $t_0(T_x \omega)=t+x$.

(b) As in (a), by~(R1) we have $-t+x\in R_{T_{-x}\omega}$, and the points of $R_{T_{-x}\omega}$ in $(-t+x,0]$ correspond to the points of $R_\omega$ in $(-t,-x]\subset(-t,0]$. There are none, by the definition of $t_0$. So $t_0(T_{-x}\omega)=t-x$.
\end{proof}

Thus along every orbit the height increases with unit speed and is reset to $0$ at each upcrossing time, and consecutive resets are at least $2\ell$ apart.

\subsection{Measurability of the height}

Since the definition of $t_0$ involves uncountably many times, its measurability requires an argument. For $\omega\in\Omega$ and $t\ge0$ write
\[
\Phi(\omega,t)=G(T_{-t}\omega)=g_\omega(-t),
\]
the value of the probe $t$ units of time in the past. The function $\Phi$ is $\Sigma\otimes\B([0,\infty))$-measurable, $0\le\Phi\le1$, and for each $\omega$ the function $t\mapsto\Phi(\omega,t)$ is $1/h$-Lipschitz by Lemma~\ref{l:lip}.

Let
\[
\sigma(\omega)=\inf\{t\ge0:\ \Phi(\omega,t)\le a\}\in[0,\infty]
\]
be the time back to the last low point, with $\sigma(\omega)=\infty$ if there is none; by continuity the infimum is attained when $\sigma(\omega)<\infty$. Let $H$ be the set of those $\omega$ with $\sigma(\omega)<\infty$ for which $\Phi(\omega,t)\ge b$ for some $t\in[0,\sigma(\omega)]$, and for $\omega\in H$ define
\[
\tau(\omega)=\max\{t\in[0,\sigma(\omega)]:\ \Phi(\omega,t)\ge b\},
\]
the maximum being attained by continuity. In words, $\omega\in H$ means that since the last time the probe was at or below the level $a$ it has already reached the level $b$, and $\tau(\omega)$ is the time elapsed since it first did so.

\begin{lemma}\label{l:meas}
The function $\sigma$ is measurable, $H\in\Sigma$, and $\tau\colon H\to[0,\infty)$ is measurable.
\end{lemma}

\begin{proof}
For $x\ge0$, by continuity,
\[
\{\sigma\le x\}=\Bigl\{\omega:\ \inf_{r\in(\Q\cap[0,x])\cup\{x\}}\Phi(\omega,r)\le a\Bigr\}\in\Sigma .
\]
Fix $y\ge0$, let $\Omega_y=\{y\le\sigma<\infty\}\in\Sigma$, and define on $\Omega_y$
\[
M_y(\omega)=\max_{t\in[y,\sigma(\omega)]}\Phi(\omega,t)=\max\Bigl\{\Phi(\omega,y),\ \Phi(\omega,\sigma(\omega)),\ \sup_{q\in\Q}f_q(\omega)\Bigr\},
\]
where $f_q(\omega)=\Phi(\omega,q)$ if $y<q<\sigma(\omega)$ and $f_q(\omega)=0$ otherwise; the second equality holds by continuity and because $\Phi\ge0$. Here $\omega\mapsto\Phi(\omega,y)$ is measurable for fixed $y$; the function $\omega\mapsto\Phi(\omega,\sigma(\omega))$ is measurable on $\Omega_y$, being the composition of the measurable map $\omega\mapsto(\omega,\sigma(\omega))$ with $\Phi$; and each $f_q$ is measurable because $\{\sigma>q\}\in\Sigma$. Hence $M_y$ is measurable on $\Omega_y$.

Now $H=\Omega_0\cap\{M_0\ge b\}\in\Sigma$. Moreover, for $\omega\in H$ and $y\ge0$, since the maximum defining $\tau(\omega)$ is attained, we have $\tau(\omega)\ge y$ if and only if $\sigma(\omega)\ge y$ and $M_y(\omega)\ge b$. Thus
\[
\{\omega\in H:\ \tau(\omega)\ge y\}=H\cap\Omega_y\cap\{M_y\ge b\}\in\Sigma\qquad\text{for every }y\ge0,
\]
and $\tau$ is measurable on $H$.
\end{proof}

\begin{lemma}\label{l:tau}
\begin{itemize}
\item[(i)] If $\omega\in H$, then $t_0(\omega)=\tau(\omega)$.
\item[(ii)] If $t_0(\omega)<\ell$, then $\omega\in H$.
\end{itemize}
\end{lemma}

\begin{proof}
(i) Write $\sigma=\sigma(\omega)$, $\tau=\tau(\omega)$ and $g=g_\omega$. Since $\Phi(\omega,\sigma)\le a<b$ we have $\tau<\sigma$, and since $\Phi(\omega,t)<b$ for $t\in(\tau,\sigma]$, continuity gives $g(-\tau)=b$. Moreover $g(-\sigma)\le a$ and $g<b$ on $[-\sigma,-\tau)$, so $-\tau\in R_\omega$ and $t_0(\omega)\le\tau$. Suppose that some $u\in(-\tau,0]$ belonged to $R_\omega$, and let $s<u$ be as in the definition. Since $g>a$ on $(-\sigma,0]$ by the definition of $\sigma$, we must have $s\le-\sigma<-\tau<u$, so $-\tau\in[s,u)$ and therefore $g(-\tau)<b$, a contradiction. Hence $t_0(\omega)\ge\tau$.

(ii) Write $t_0=t_0(\omega)$. Since $-t_0$ is an upcrossing time, there is $s<-t_0$ with $g_\omega(s)\le a$, so $\sigma(\omega)\le-s<\infty$. By Lemma~\ref{l:lip}, $g_\omega\ge b-t_0/h>b-\ell/h=a$ on $[-t_0,0]$, so $\sigma(\omega)>t_0$. As $\Phi(\omega,t_0)=g_\omega(-t_0)=b$, we conclude that $\omega\in H$.
\end{proof}

\begin{proposition}\label{p:t0meas}
For every $\omega\in\Omega$,
\[
t_0(\omega)=\inf\bigl\{\,q+\tau(T_{-q}\omega):\ q\in\Q,\ q\ge0,\ T_{-q}\omega\in H\,\bigr\},
\]
with the convention $\inf\varnothing=\infty$. In particular, $t_0$ is measurable.
\end{proposition}

\begin{proof}
If $q\ge0$, $T_{-q}\omega\in H$ and $v=\tau(T_{-q}\omega)$, then $-v\in R_{T_{-q}\omega}=R_\omega+q$ by Lemma~\ref{l:tau}(i) and~(R1); hence $-(q+v)\in R_\omega$ and $q+v\ge t_0(\omega)$. Conversely, suppose that $t=t_0(\omega)<\infty$ and choose rational $q\in[0,t]$ with $q>t-\ell$. By Lemma~\ref{l:saw}(b), $t_0(T_{-q}\omega)=t-q<\ell$, so by Lemma~\ref{l:tau} we have $T_{-q}\omega\in H$ and $q+\tau(T_{-q}\omega)=t$. Measurability follows because, for each $q$, the function equal to $q+\tau(T_{-q}\omega)$ on the measurable set $\{\omega:\ T_{-q}\omega\in H\}$ and to $\infty$ elsewhere is measurable.
\end{proof}

\subsection{Distribution of the height}

\begin{lemma}\label{l:uniform}
Let $\lambda=\mu(\{t_0<\ell\})/\ell$. Then $\lambda>0$, and
\[
\mu(\{t_0\in[p,q)\})=\lambda(q-p)\qquad\text{whenever }0\le p<q\le\ell .
\]
\end{lemma}

\begin{proof}
Let $0\le p<q$ and $x\ge0$ with $q+x\le\ell$. By Lemma~\ref{l:saw}(a), $T_x$ maps $\{t_0\in[p,q)\}$ into $\{t_0\in[p+x,q+x)\}$, and by Lemma~\ref{l:saw}(b), $T_{-x}$ maps the latter set into the former. Hence $T_x$ restricts to a bijection between these two sets, and since $T_x$ preserves $\mu$ they have the same measure. Consequently the nondecreasing function $\Lambda(y)=\mu(\{t_0\in[0,y)\})$, $0\le y\le\ell$, satisfies $\Lambda(y_1+y_2)=\Lambda(y_1)+\Lambda(y_2)$ whenever $y_1+y_2\le\ell$, and therefore $\Lambda(y)=\lambda y$ on $[0,\ell]$. This gives the formula.

To see that $\lambda>0$, let $\omega\in P$ be such that $R_\omega\ne\varnothing$; by~(R3) this holds for a.e.\ $\omega\in P$. Pick $t_1\in R_\omega$. By~(R1), $0\in R_{T_{t_1}\omega}$, so $t_0(T_{t_1}\omega)=0$, and Lemma~\ref{l:saw}(a) gives $t_0(T_q\omega)=q-t_1<\ell$ for every rational $q\in[t_1,t_1+\ell)$. Hence, up to a null set,
\[
P\subset\bigcup_{q\in\Q}T_q^{-1}\bigl(\{t_0<\ell\}\bigr),
\]
and since $\mu(P)>0$ and each $T_q$ preserves $\mu$, we obtain $\mu(\{t_0<\ell\})>0$.
\end{proof}

\subsection{The extremal set and the end of the proof}\label{ss:extremal}

Fix $\alpha\in(0,1)$ and put
\[
\varepsilon=\alpha\ell,\qquad \delta=\ell-\varepsilon=\varepsilon\Bigl(\frac1\alpha-1\Bigr),
\]
so that $\delta+\varepsilon=\varepsilon/\alpha=\ell$. Consider the band of heights $[\delta,\ell)$ and the halo $(0,\delta)$ below it:
\[
E=\{t_0\in[\delta,\ell)\},\qquad Z=\{t_0\in(0,\delta)\}.
\]
Both sets are measurable by Proposition~\ref{p:t0meas}, they are disjoint, and by Lemma~\ref{l:uniform}
\[
\mu(E)=\lambda\varepsilon>0,\qquad\mu(Z)=\lambda\delta .
\]

We claim that $E\cup Z\subset\{T^*\chi_E>\alpha\}$. Let $t_0(\omega)=t\in(0,\ell)$. For $u\in[0,\ell-t)$ we have $t+u<\ell<2\ell$, so $t_0(T_u\omega)=t+u$ by Lemma~\ref{l:saw}(a). Hence, for such $u$,
\[
T_u\omega\in E\iff u\ge\delta-t .
\]
If $t\in[\delta,\ell)$, then $T_u\omega\in E$ for all $u\in[0,\ell-t)$, so $\frac1s\int_0^s\chi_E(T_u\omega)\,du=1>\alpha$ for every $s\in(0,\ell-t]$. If $t\in(0,\delta)$, take $s=\ell-t$; the orbit spends the time $[0 , \delta-t)$ outside $E$ and then the time $[\delta-t,\ell-t)=[\delta-t,\delta+\varepsilon-t)$ inside $E$, so
\[
\frac1s\int_0^s\chi_E(T_u\omega)\,du=\frac{\varepsilon}{\ell-t}>\frac{\varepsilon}{\ell}=\alpha .
\]
This proves the claim. Consequently
\[
\mu(\{T^*\chi_E>\alpha\})\ge\mu(E)+\mu(Z)=\lambda(\varepsilon+\delta)=\lambda\ell=\frac{\lambda\varepsilon}{\alpha}=\frac{\mu(E)}{\alpha},
\]
which is~\eqref{e:goal}. This completes the proof of case~(b), and hence of Theorem~\ref{t:main}. \qed

\begin{remark}\label{r:extremal}
\begin{enumerate}
\item Points of height exactly $0$ are excluded from $Z$, since for them the average used above equals $\alpha$; this costs nothing, as $\mu(\{t_0=0\})=0$ by Lemma~\ref{l:uniform}.
\item Since $\lambda\le1/\ell$, the extremal set satisfies $\mu(E)=\lambda\varepsilon\le\alpha<1$; this gives Remark~\ref{r:main}(3).
\item Along each orbit, the pair $(E,Z)$ reproduces the Euclidean extremal configuration of Remark~\ref{r:main}(2): an interval of length $\varepsilon$ preceded by its halo of length $\varepsilon(1/\alpha-1)$. The upcrossing times provide the room needed to place one such copy between consecutive resets of the height.
\end{enumerate}
\end{remark}

\section{Further remarks}\label{s:remarks}

\subsection{The two-sided maximal operator}

Let
\[
T^{\pm}f(\omega)=\sup_{\substack{a',b'\ge0\\ a'+b'>0}}\frac{1}{a'+b'}\int_{-a'}^{b'}|f(T_u\omega)|\,du ,
\]
and let $C^{\pm}(\alpha)$ be defined as $C(\alpha)$ with $T^*$ replaced by $T^{\pm}$.

\begin{proposition}\label{p:twosided}
In case~(a) of Theorem~\ref{t:main}, $C^{\pm}(\alpha)=1$ for every $\alpha\in(0,1)$; in case~(b), $C^{\pm}(\alpha)=2/\alpha-1$ for every $\alpha\in(0,1)$.
\end{proposition}

\begin{proof}
Let $T^*_-$ be the one-sided maximal operator of the reversed flow,
\[
T^*_-f(\omega)=\sup_{s>0}\frac1s\int_{-s}^0|f(T_u\omega)|\,du .
\] For $a',b'>0$ the average over $[-a',b']$ is a convex combination of the averages over $[-a',0]$ and over $[0,b']$. Hence $\{T^{\pm}\chi_E>\alpha\}\subset L_+\cup L_-$, where $L_+=\{T^*\chi_E>\alpha\}$ and $L_-=\{T^*_-\chi_E>\alpha\}$. By Theorem~\ref{t:max} applied to both flows, $\mu(L_\pm)\le\mu(E)/\alpha$, and by Lemma~\ref{l:diff} applied to both flows, $L_+\cap L_-$ contains a.e.\ point of $E$. Therefore
\[
\mu(\{T^{\pm}\chi_E>\alpha\})\le\mu(L_+)+\mu(L_-)-\mu(L_+\cap L_-)\le\Bigl(\frac2\alpha-1\Bigr)\mu(E),
\]
so $C^{\pm}(\alpha)\le2/\alpha-1$. Taking $E=\Omega$ gives $C^{\pm}(\alpha)\ge1$, and in case~(a) the argument of Section~\ref{ss:casea} gives $T^{\pm}\chi_E=\chi_E$ a.e., hence $C^{\pm}(\alpha)=1$.

In case~(b) we keep the notation of Section~\ref{s:proof} and set $\varepsilon=\alpha\ell/(2-\alpha)$ and $\delta=\varepsilon(1/\alpha-1)$, so that $2\delta+\varepsilon=\ell$ and $\delta+\varepsilon=\varepsilon/\alpha$. Let
\[
E=\{t_0\in[\delta,\delta+\varepsilon)\},\qquad Z_-=\{t_0\in(0,\delta)\},\qquad Z_+=\{t_0\in(\delta+\varepsilon,\ell)\}.
\]
The argument of Section~\ref{ss:extremal}, with forward windows ending at the top of the band, shows that $E\cup Z_-\subset\{T^{\pm}\chi_E>\alpha\}$. If $t_0(\omega)=t\in(\delta+\varepsilon,\ell)$, Lemma~\ref{l:saw}(b) gives $t_0(T_{-x}\omega)=t-x$ for $0\le x\le t$, so on the backward window $[-(t-\delta),0]$ the orbit spends the time $\varepsilon$ in $E$, and the average equals $\varepsilon/(t-\delta)>\varepsilon/(\delta+\varepsilon)=\alpha$, because $t-\delta<\ell-\delta=\delta+\varepsilon$. Hence $Z_+\subset\{T^{\pm}\chi_E>\alpha\}$ as well, and by Lemma~\ref{l:uniform}
\[
\mu(\{T^{\pm}\chi_E>\alpha\})\ge\lambda(\varepsilon+2\delta)=\lambda\ell=\Bigl(\frac2\alpha-1\Bigr)\lambda\varepsilon=\Bigl(\frac2\alpha-1\Bigr)\mu(E). \qedhere
\]
\end{proof}

\subsection{Relation with the Ambrose--Kakutani representation}

The set $\{\omega:\ 0\in R_\omega\}$ is a cross-section for the flow, with return times at least $2\ell$, and $t_0$ is the vertical coordinate of the corresponding representation of the flow as a flow under a function; compare Ambrose~\cite{Amb41} and Ambrose--Kakutani~\cite{AK42}. We have avoided constructing the transverse measure on this (null) cross-section by computing the distribution of $t_0$ directly in Lemma~\ref{l:uniform}. In particular, no standard-space assumption is needed, and periodic flows require no separate treatment.

\subsection{Multiparameter flows}

Problem~7.1 of~\cite{HP16} also concerns measure-preserving actions of $\R^n$, that is, $n$ commuting flows, together with maximal operators associated with rectangles or balls containing the origin. It would be interesting to know whether a continuous-parameter version of the transference argument of~\cite[Section~3]{HP16} yields the Euclidean estimates~\eqref{e:euclid} for these operators, thereby avoiding the loss in the exponent of $C^*_{HL}$ in~\eqref{e:HP16} that comes from discretization, and whether, in analogy with Theorem~\ref{t:main}(b), the resulting bounds are sharp for free actions.

\subsection*{Acknowledgments}
The author thanks Paul Hagelstein for suggesting this problem and for helpful discussions.


\end{document}